\documentclass[11pt]{article}

\usepackage{amsmath}
\usepackage{amsfonts}
\usepackage{amsthm}
\usepackage{amssymb}
\usepackage[english]{babel}
\usepackage{latexsym}
\usepackage[hmargin=2.5cm,vmargin=2.3cm]{geometry}
\usepackage{enumerate,comment}
\usepackage{graphicx}
\usepackage{nicefrac}
\usepackage{mathrsfs}
\usepackage[affil-it]{authblk}
\usepackage{natbib}
\usepackage{framed}
\usepackage[affil-it]{authblk}
\usepackage{multirow}
\usepackage{color}
\usepackage{multirow}
\usepackage{cancel}
\usepackage{epstopdf}
\usepackage{hyperref}
\usepackage[normalem]{ulem}

\theoremstyle{plain}
\newtheorem{theorem}{Theorem}[section]
\newtheorem{lemma}[theorem]{Lemma}
\newtheorem{proposition}[theorem]{Proposition}
\newtheorem{corollary}[theorem]{Corollary}
\theoremstyle{definition}

\newtheorem{remark}[theorem]{Remark}

\newcommand{\E}{\mathbb E}

\newcommand{\N}{\mathbb N}

\newcommand{\R}{\mathbb R}

\def \qmo{``}

\def \qmcsp{'' }

\def\keywords{\vspace{.5em}
{\noindent\textbf{Keywords}:\,\relax%
}}

\makeatletter
\let\@fnsymbol\@arabic
\makeatother

\begin{document}

\title{Inter-order relations between moments of a Student $t$ distribution, with an application to $L_p$-quantiles}

\author{Valeria Bignozzi\thanks{\mbox{Department of Statistics and Quantitative Methods, University of Milano-Bicocca, Italy \texttt{valeria.bignozzi@unimib.it}}}\,,~Luca Merlo\thanks{{Department of Human Sciences, European University of Rome, Italy \texttt{luca.merlo@unier.it}}}\,,~Lea Petrella\thanks{\mbox{MEMOTEF Department, Sapienza University of Rome, Italy \texttt{lea.petrella@uniroma1.it}}}}

\date{\today}

\maketitle

\parindent 0em \noindent

\begin{abstract}
 This paper introduces inter-order formulas for partial and complete moments of a Student $t$ distribution with $n$ degrees of freedom. 
 We show how the partial moment of order $n-j$ about any real value $m$ can be expressed in terms of the partial moment of order $j-1$ for $j$ in $\{1,\ldots, n\}$. Closed form expressions for the complete moments are also established. We then focus on $L_p$-quantiles, which represent a class of generalized quantiles defined through an asymmetric $p$-power loss function. Based on the results obtained, we also show that for a Student $t$ distribution the $L_{n-j+1}$-quantile and the $L_{j}$-quantile coincide at any confidence level $\tau$ in $(0,1)$.

\end{abstract}

\section{Introduction}
Partial moments play an important role in several areas of statistics and applied mathematics, including 
 reliability modeling (\citealp{lata1983moments}), {insurance pricing (\citealp{Denuit02}) and 
financial risk theory (\citealp{harlow1989asset}). \cite{bawa1975optimal} was among the first ones to introduce partial moments in financial economics, 
\cite{ANTHONISZ20122122} and \cite{yao2021general} subsequently investigated the use of (lower) partial moments in the context of asset pricing and portfolio optimization, respectively.
 \cite{unser2000lower} examined investors' risk perception in financial decision making using lower partial moments as a measure of perceived risk, while \cite{antle2010asymmetry} considered partial moments as a flexible way to estimate the effects of biophysical and management processes on agricultural output distributions.
 From a theoretical standpoint, their properties and recursive relationships have been investigated in, for instance, \cite{winkler}, where several methods for determining partial moments are discussed or in \cite{abraham} who characterized different continuous distributions belonging to the exponential and Pearson families based on partial moments.
 In this work, we are concerned with partial and complete moments of the Student $t$ distribution with arbitrary degrees of freedom $n\in\mathbb{N}$. In both frequentist and Bayesian paradigms, the Student $t$ has thoroughly pervaded the statistical literature in those situations where heavy tails or outliers arise as a robust alternative to the normal distribution (\citealp{jonsson2011heavy}). These characteristics make it suitable for the analysis of data in the economic sciences and risk analysis; see, for instance, \cite{lange1989} and \cite{mcneil2015}. In particular,
 results on the moments of a Student $t$ distribution with $n$ degrees of freedom have been obtained 
 in \cite{winkler} where it was shown that the $j$-th order partial moment about the origin can be expressed in terms of the $(j-2)$-th order partial moment, provided that $j < n$. More recently, \cite{BernardiBPStudent} presented a recurrence formula by substituting repeatedly the partial moment of order $j-2$ in the characterization of the $j$-th order partial moment illustrated in \cite{TableofInt07}. 
 In the literature, general expressions for central moments of skewed, truncated or folded Student $t$ distributions have also been derived by \cite{nadarajah2004skewed} and \cite{kim2008moments}, for example. However, relationships for partial and complete higher-order moments about any arbitrary point of a Student $t$ distribution have not yet been determined. 
\\
This paper contributes to the existing literature in two ways. First, we show that for a Student $t$ distribution with $n $ degrees of freedom, the partial moment of order $n-j$ about any point $m\in\R$ can be expressed in terms of the partial moment of order $j-1$ times a multiplicative constant that depends on $m, n$ and $j\in\{1,\ldots, n\}$. 
 In addition, closed form expressions for the complete moments 
 are also derived in terms of the Gamma and the hypergeometric functions. The proofs of these results involve concepts from combinatorial analysis as well as exploit fundamental properties of known special functions. 
 To the best of our knowledge, this is the first time that inter-order relationships for partial and complete moments of the Student $t$ distribution are investigated, giving a simple way for calculating higher order moments from those of lower order.

The second contribution of the paper establishes inter-order relations of equality between 
 $L_p$-quantiles of the Student $t$ distribution. In the literature, $L_p$-quantiles, introduced by \cite{Chen96} as a natural extension of quantiles, are an important class of statistical functionals defined as the minimizers of an expected asymmetric $p$-power loss function. Indeed, for $p=1$ and for $p=2$, they correspond respectively to quantiles and expectiles, which have been proposed by \cite{NeweyP87}, based on asymmetric least-squares estimation as a \qmo quantile-like\qmcsp generalization of the mean. Basic properties of the $L_p$-quantiles have been given by \cite{Chen96} and, in the context of generalized quantiles, \cite{Remillard95} provided sufficient conditions under which $L_1$-quantiles (quantiles) and $L_2$-quantiles (expectiles) coincide. 
 More generally, both quantiles and expectiles, and in turn $L_p$-quantiles, can be embedded in the wider class of M-quantiles of \cite{BrecklingChambers88} which extend the ideas of M-estimation of \cite{huber1964robust} by introducing asymmetric influence functions to model the entire distribution of a random variable. Generalized quantile models have been implemented in a broad range of applications, such as multilevel modeling (\citealt{alfo2021m} and \citealt{merlo2021quantile}), nonparametric regression (\citealt{Pratesi09}), multivariate analysis (\citealt{merlo2022marginal}), 
 economics and finance (\citealt{Gneiting11}, \citealt{Bellinietal14} and \citealt{daouia2019extreme}). In the latter context, 
 $L_p$-quantiles 
have received particular consideration 
as potential competitors to the most used risk measures in banking and insurance, namely Value at Risk and Expected Shortfall. 
 Indeed, besides being elicitable (\citealt{Lambertetal08}), they possess several interesting properties in terms of risk measures (see for instance \citealt{bellini2012isotonicity} and \citealt{Bellinietal14}). Moreover, when $p=2$, $L_2$-quantiles are the only risk measure that is both coherent (\citealt{Artzner99}) and elicitable (please see also \cite{Bellinietal14} and \cite{ziegel2016coherence} for a detailed analysis of expectiles as a risk measure).

In this paper, by exploiting the results derived on partial moments we show that for a Student $t$ distribution with $n$ degrees of freedom, the $L_{n-j+1}$-quantile and the  $L_{j}$-quantile coincide for any $\tau\in(0,1)$ (and the same holds for any affine transformation). This result generalizes the one in \cite{Koenker93} who showed that the Student $t$ distribution with 2 degrees of freedom, or any affine transformation thereof, is the unique distribution where $L_1$-quantiles and $L_2$-quantiles match each other, and in \cite{BernardiBPStudent} who proved that the $L_n$-quantiles and $L_1$-quantiles coincide for any $\tau \in (0,1)$.

With this paper, we give new insights on the properties of the Student $t$ distribution and, at the same time, we contribute in the context of risk management to 
 the framework of \cite{belliniDB17}, \cite{li2022pelve} and \cite{fiori2022generalized} for comparing general pairs of risk measures and determining when these coincide.\\

The rest of the paper is organized as follows. Section \ref{sec:bg} introduces the notation and the mathematical background that will be used in our analysis. In Section \ref{sec:t-distr} we present the first important result of the paper deriving inter-order formulas for partial and complete moments of arbitrary order for a Student $t$ distribution. Finally, in Section \ref{seq:lp-quant} we introduce the $L_p$-quantiles and provide a symmetry characterization of $L_p$-quantiles of the Student $t$ distribution. The proofs of the main results are collected in the \nameref{Appendix}.

\section{Mathematical background}\label{sec:bg}
We start by introducing some basic mathematical notation. Let $\mathbb{N}:=\{1,2,\ldots\}$ denote the set of positive integers, $\mathbb{N}_0=\N\cup \{0\}$ and $\R^+:=(0,+\infty)$ the positive real line; $n$ will always be  an element of $\N$. The binomial coefficient if defined by $\binom{p}{q}=\frac{p!}{q!(p-q)!},$ for $p,q\in\N_0$, where  $n!=n\cdot(n-1)\cdot\ldots\cdot 1,$ is the factorial with the convention $0!=1$. We denote $(x)_+=\max(x,0)$ and $(x)_-=\max(-x,0)$ the positive and negative part of $x\in\R$, respectively. Given an atomless probability space $(\Omega,\mathcal{F},\mathbb{P})$, $L^p:=L^p(\Omega,\mathcal{F},\mathbb{P})$ represents the space of random variables with finite $p$-moment, $p\in[1,\infty)$. $L^0:=L^0(\Omega,\mathcal{F},\mathbb{P})$ and $L^{\infty}:=L^\infty(\Omega,\mathcal{F},\mathbb{P})$ denote, respectively, the space of measurable and bounded random variables. Given a continuous real-valued random variable $Y:\Omega\to \R$, we denote $f_Y$ and $F_Y$ its density and distribution functions. To improve readability the subscript may be omitted when no confusion is likely to arise. 
{For $Y\in L^k(\Omega,\mathcal{F},\mathbb{P}),~k\in\N_{0}$, the upper and lower partial moments of order $k$ about any point $m\in\R$ are denoted as follows
\begin{equation}\label{eq:uppermom}
\E[((Y-m)_+)^{k}] = \int_{m}^{+\infty} (y-m)^k \textrm{d} F_Y (y)
\end{equation}
and
\begin{equation}\label{eq:lowermom}
\E[((Y-m)_-)^{k}] = \int_{-\infty}^m (m-k)^k \textrm{d} F_Y (y) = (-1)^k \int_{-\infty}^m (y-m)^k \textrm{d} F_Y (y).
\end{equation}}


In this paper we make extensive use of some special functions: the Gamma, the Beta and the hypergeometric functions. These functions are widely discussed in every book on special functions. Among many others we cite \cite{abramowitz1972handbook}, \cite{Askey75}, \cite{temme1996special} and \cite{mathai2008special}. Here, we briefly introduce the hypergeometric function with its most relevant properties.

To this end, we first recall the notion of Pochhammer's symbol:
$$
(x)_k:=(x+k-1)(x+k-2)\cdot\ldots\cdot(x+1)x, \quad x\in\mathbb{R},~k\in\N_0,
$$
{with the convention $(x)_0 = 1$.} Note that for $x\in\R^+$, one has $(x)_k=\frac{\Gamma(x+k)}{\Gamma(x)}$, where $\Gamma(x)$ denotes the Gamma function of the positive real number $x$. For $m\in\N_0$, it is immediate to verify
$$
(-m)_k=\left\{ 
\begin{array}{cc}
0\qquad&\textrm{if }k>m,\\
(-1)^k\frac{m!}{(m-k)!}\qquad&\textrm{if } k\leq m.
\end{array}%
\right. 
$$
 Then, the hypergeometric function ${_{2}}F_{1}(a,b,c;z)$ is defined by the series
\begin{equation}\label{eq:2f1}
{_{2}}F_{1}(a,b,c;z)=\sum_{k=0}^\infty\frac{(a)_k(b)_k}{(c)_k}\frac{z^k}{k!},
\end{equation}
with $c\not\in\{0,-1,-2,\dots\}$.
The radius of convergence of the series is the unity, but  it may be extended under weak conditions. For instance, when $a,b,c,z\in\R$ the analytic continuation of the hypergeometric function can be obtained when $c>b>0$ (or $c>a>0$) and $z<1$, see \cite{temme1996special}. The series terminates if either $a$ or $b$ is a nonpositive integer, in which case the function reduces to a polynomial {and the convergence is guaranteed for any $z\in\R$}. Indeed, when $a$ or $b$ is equal to $-m$, with $m \in\mathbb{N}_0$, the function ${_{2}}F_{1}(a,b,c;z)$ is a polynomial of degree $m$ in $z$, that is:
\begin{equation*}
{_{2}}F_{1}(-m,b,c;z)=\sum_{k=0}^m\frac{(-m)_k(b)_k}{(c)_k}\frac{z^k}{k!}=\sum_{k=0}^m (-1)^k\binom{m}{k}\frac{(b)_k}{(c)_k}z^k.
\end{equation*}

{The hypergeometric function satisfies a great variety of relations. First we observe that it is symmetric in $a$ and $b$, giving ${_{2}}F_{1}(a,b,c;z) = {_{2}}F_{1}(b,a,c;z)$. Another} useful property which will be often used throughout the paper is the following Euler's transformation (\citealt{temme1996special}, p. 110). 
For $c>a>0$ {(or equivalently $c>b>0$ from the symmetry in ${_{2}}F_{1}(a,b,c;z)$ with respect to $a$ and $b$) and $|\mbox{arg}(1-z)| < \pi$}:
\begin{equation}\label{Euler}
{_{2}}F_{1}(a,b,c;z)=(1-z)^{c-a-b}{_{2}}F_{1}(c-a,c-b,c;z),
\end{equation}
where $\mbox{arg}(z)$ denotes the argument of the real number $z$.

The last property of the hypergeometric function that we need is:
{\small
\begin{align}
&{_{2}{F}_1}\left(a,b,c;z\right)=\label{property}\\&
=\frac{\Gamma(c)\Gamma(b-a)}{\Gamma(b)\Gamma(c-a)}(-z)^{-a}{_{2}{F}_1}\left(a,a-c+1,a-b+1;\frac{1}{z}\right)+\frac{\Gamma(c)\Gamma(a-b)}{\Gamma(a)\Gamma(c-b)}(-z)^{-b}{_{2}{F}_1}\left(b-c-1,b,b-a+1;\frac{1}{z}\right),\notag
\end{align}
}for $a-b \notin \mathbb{Z}$, $|\mbox{arg}(-z)| < \pi$ and the Gamma functions well defined (
\citealt{temme1996special}, p. 120). 
In this contribution we will briefly encounter  the \textit{generalized hypergeometric function}  defined as
\begin{equation*}
{_{p}}F_{q}(\alpha_1,\alpha_2,\ldots,\alpha_p,\beta_1,\beta_2,\ldots,\beta_q;z)=\sum_{k=0}^\infty\frac{(\alpha_1)_k(\alpha_2)_k\cdot\ldots\cdot (\alpha_p)_k}{(\beta_1)_k(\beta_2)_k\cdot\ldots\cdot(\beta_q)_k}\frac{z^k}{k!},\qquad {\textrm{for } p,q\in\N, p\leq q+1.}
\end{equation*}
{The radius of convergence of this series is unity for $p=q+1$, while it is $\infty$ for $p<q+1$.} 
The parameters $\alpha_1,\alpha_2,\ldots,\alpha_p$ can be commuted as well as the parameters $\beta_1,\beta_2,\ldots,\beta_q$. Furthermore if $\alpha_j=\beta_k$ for some $j\in\{1,\ldots,p\}$ and 
$k\in\{1,\ldots,q\}$, then $${_{p}}F_{q}(\alpha_1,\ldots,\alpha_j,\ldots\alpha_p,\beta_1,\ldots,\beta_k,\ldots,\beta_q;z)={_{p-1}}F_{q-1}(\alpha_1,\ldots,\alpha_{j-1},\alpha_{j+1},\ldots\alpha_p,\beta_1,\ldots,\beta_{k-1},\beta_{k+1},\ldots,\beta_q;z).$$

\section{Relations between moments of the Student $t$ distribution}\label{sec:t-distr}
In this section we present innovative results for the partial and complete moments of the Student $t$ distribution. 
 Let $Y$ be a random variable with  Student $t$ distribution with $n\in\mathbb{N}$ degrees of freedom, i.e., $Y\in L^{n-1}(\Omega,\mathcal{F},\mathbb{P})$, with density function given by 
\begin{equation*}
f_{Y}(y)=\frac{\Gamma\left(\frac{n+1}{2}\right)}{\sqrt{n\pi}\Gamma\left(\frac{n}{2}\right)}\left(1+\frac{y^2}{n}\right)^{-\frac{n+1}{2}},\quad\textrm{where}\quad y\in\R.%
\end{equation*} 
In order to introduce our results, we recall the expressions for the moments of $Y$ centered around zero. Specifically, the raw moments are well defined for any  order $j\in \{0,\ldots, n-1\}$ and can be calculated as
\begin{equation}\label{eq:rawmoments}
 \E[Y^j]=\left\{ 
\begin{array}{cc}
0\qquad&\textrm{if } j \textrm{ is odd},\\
\frac{\Gamma\left(\frac{j+1}{2}\right)\Gamma\left(\frac{n-j}{2}\right)}{\sqrt{\pi}\Gamma\left(\frac{n}{2}\right)}n^{\frac{j}{2}}\qquad&\textrm{if } j \textrm{ is even}.%
\end{array}%
\right. 
\end{equation}
The raw moments of odd order are null because the density function $f_Y$ is an even function. 
 
 In the following, a closed form expression for the complete moments of order $j\in\{0,\ldots,{n-1}\}$ centered in $m\in\mathbb{R}$ is derived. 
 This result can be found in the working paper by \cite{kirkby19} based on the unpublished note by \cite{winkelbauer12}. Here, we derive the same formulas 
 with a different proof.
\begin{proposition}\label{Moments} Let $Y$ be a random variable with Student $t$ distribution with $n\in\mathbb{N}$ degrees of freedom. For  $j\in\{0,\ldots,{n-1}\}$ and $m\in\R$:
\begin{equation}\label{eq:cenmom}
\E[(Y-m)^j]=\left\{ 
\begin{tabular}{ll}
$-2mn^{\frac{j-1}{2}}\frac{\Gamma(\frac{j}{2}+1)\Gamma(\frac{n-j+1}{2})}{\sqrt{\pi}\Gamma(\frac{n}{2})}{_2}F_1\left(\frac{1-j}{2},\frac{n-j+1}{2},\frac{3}{2};-\frac{m^2}{n}\right)$& \qquad if $j~\textrm{is odd}$,\\
$n^{\frac{j}{2}}\frac{\Gamma(\frac{j+1}{2})\Gamma(\frac{n-j}{2})}{\sqrt{\pi}\Gamma(\frac{n}{2})}{_2}F_1\left(-\frac{j}{2},\frac{n-j}{2},\frac{1}{2};-\frac{m^2}{n}\right)$& \qquad if $j~\textrm{is even},$\\
\end{tabular}
\right. 
\end{equation}
where $_{2}{F}_1 (a, b, c; z)$ is the hypergeometric function defined in \eqref{eq:2f1}.
\end{proposition}
\begin{proof}
See the Appendix \ref{Asec:Moments}.
\end{proof}
\begin{remark}Note that  $_{2}{F}_1 (a, b, c; 0)=1$, therefore Proposition \ref{Moments} for $m=0$ returns the expression for the raw moments of order $j$ provided in \eqref{eq:rawmoments}. Moreover, for both odd and even $j$ 
 the first term of the hypergeometric function is {a nonpositive integer}, implying that {the function is a polynomial and the convergence is guaranteed for any $m\in\R$ and $n\in\mathbb{N}$}.\end{remark}
%

A useful property of upper and lower partial moments of the Student $t$ distribution that will be needed for the main theorems can be proved by the following lemma.
\begin{lemma}\label{lemma}Let $Y$ be a random variable with Student $t$ distribution with $n\in\mathbb{N}$ degrees of freedom and let $j\in\{1,\ldots,{n}\}$:
\begin{enumerate}\item[(1)] For $m\geq 0$ the following equation holds:
\begin{equation}\label{eq:mpos}
\E[((Y-m)_+)^{n-j}]=(m^2+n)^{\frac{n-2j+1}{2}}\E[((Y-m)_+)^{j-1}].
\end{equation}
\item[(2)] For $m<0$
\begin{equation}\label{eq:mneg}
\E[((Y-m)_-)^{n-j}]=(m^2+n)^{\frac{n-2j+1}{2}}\E[((Y-m)_-)^{j-1}].
\end{equation}

\end{enumerate}
\end{lemma}
\begin{proof} 
For the proof see Appendix \ref{Asec:lemma}.
\end{proof}
Thanks to Lemma \ref{lemma}, in the following theorem we derive an interesting link between
 the central moments of order $n-j$ with that of order $j-1$ for any $m\in\R$.
\begin{theorem}\label{th:moments} Let $Y$ be a random variable with Student $t$ distribution with $n\in\mathbb{N}$ degrees of freedom. For $j\in\{1,\ldots,n\}$ and $m\in\R$, the following equation for the central moments about the point $m$ holds:
\begin{enumerate}\item For $n$ odd:
\begin{equation}\label{eq:moments}
\mathbb{E}[(Y-m)^{n-j}]=(m^2+n)^{\frac{n-2j+1}{2}} \mathbb{E}[(Y-m)^{j-1}].
\end{equation}
\item For $n$ even:
\begin{align}\label{eq:momentseven}
\mathbb{E}[(Y-m)^{n-j}]&={(m^2+n)^{\frac{n-2j+1}{2}}\left[-\mathbb{E}[(Y-m)^{j-1}]+2\mathbb{E}[((Y-m)_+)^{j-1}] \right]}.
\end{align}
\end{enumerate}
\end{theorem}
\begin{proof}For the proof see Appendix \ref{Asec:thmoments}.
\end{proof}
Theorem \ref{th:moments} provides an original inter-order relation between moments of the Student $t$ distribution. 
\begin{remark}
It is worth noting that Theorem \ref{th:moments} also establishes an expression for the raw moments of $Y$. Indeed, for $m=0$, we have that
\begin{equation}\label{eq:cenmom2}
\E[Y^{n-j}]=\left\{ 
\begin{tabular}{ll}
$n^\frac{n-2j+1}{2} \E[Y^{j-1}]$& \qquad if $n$ and $j$ are odd,\\
$n^\frac{n-2j+1}{2} \E[|Y|^{j-1}]$& \qquad if $n$ and $j$ are even,\\
$0$ & \qquad otherwise.
\end{tabular}
\right. 
\end{equation}
\end{remark}

 In the next results, we show that a similar relation holds also for the upper and lower partial moments.
\begin{theorem}\label{th:momentspp}Let ${Y}$ be a random variable with Student $t$ distribution with $n\in\mathbb{N}$ degrees of freedom.
 For $j\in\{1,\ldots,n\}$ and $m\in\R$,  the following equation for the upper partial moments holds:
\begin{equation}\label{eq:momentsp}
\mathbb{E}[((Y-m)_+)^{n-j}]=(m^2+n)^{\frac{n-2j+1}{2}}\mathbb{E}[((Y-m)_+)^{j-1}].
\end{equation}

\end{theorem}\begin{proof}
For $m\geq 0$ the result follows from Lemma \ref{lemma}(1).
For $m<0$, we recall that the upper partial moment of order $ j\in\{0,\ldots,{n-1}\}$ can be  decomposed as:
$$
\mathbb{E}[((Y-m)_+)^{j}]=\mathbb{E}[(Y-m)^{j}]-(-1)^j\mathbb{E}[((Y-m)_-)^{j}].
$$
The result follows from Lemma \ref{lemma}(2) together with  Theorem \ref{th:moments} (1) and (2) for the case $n$ odd and $n$ even, respectively.
\end{proof}
\begin{corollary}\label{partial moments}
Let $Y$ be a random variable with Student $t$ distribution with $n\in\mathbb{N}$ degrees of freedom.
 For  $j\in\{1,\ldots,n\}$ and $m\in\R$,  the following equation for the lower partial moments holds:
 \begin{equation}\label{eq:pmoments}
\E[((Y-m)_-)^{n-j}] = (m^2+n)^{\frac{n-2j+1}{2}}\E[((Y-m)_-)^{j-1}].
\end{equation}
\end{corollary}
\begin{proof}
The proof follows from Theorem \ref{th:moments} and \ref{th:momentspp}, using the equality:
 $$\E[(Y-m)^j]=\E[((Y-m)_+)^j]+(-1)^j\E[((Y-m)_-)^j],\quad \textrm{for } j\in\{0,\ldots,{n-1}\}.$$
\end{proof}
From another perspective, Theorem \ref{th:momentspp} and Corollary \ref{partial moments} assert 
 that for a Student $t$ distribution with $n$ degrees of freedom the ratio between the upper (lower) partial moment of order $n-j$ and that of order $j-1$ is a deterministic function of $m$. Figure \ref{fig:constant} depicts the function $(m^2+n)^{\frac{n-2j+1}{2}}$ for various values of $n$ and $j\in\{1,\ldots,n\}$. As it can be observed, this is an even convex function of $m$ when the exponent $\frac{n-2j+1}{2}$ is positive and concave otherwise. We can also see that in both cases the global minimum (maximum) is at $m=0$. Evidently, if $j = \frac{n+1}{2}$ and $n$ is odd, this is a constant function equal to 1.

\begin{figure}
\centering
\includegraphics[width=.495\textwidth]{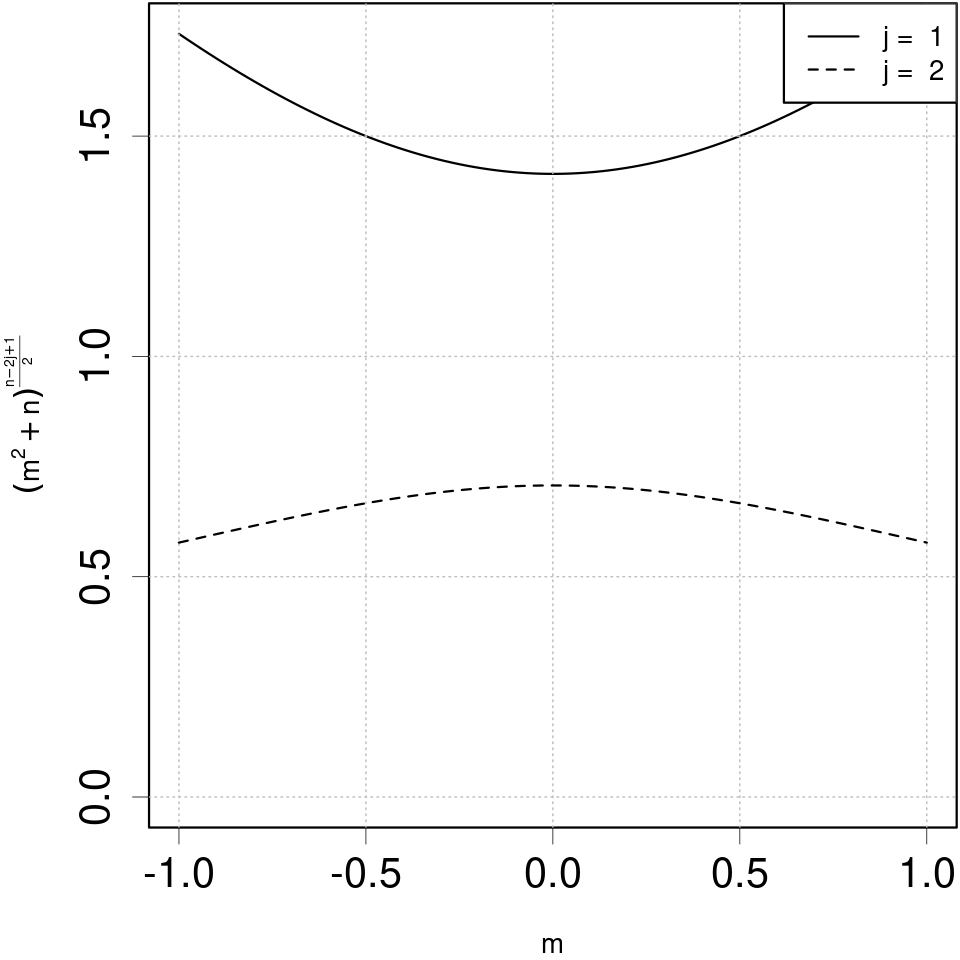}
\includegraphics[width=.495\textwidth]{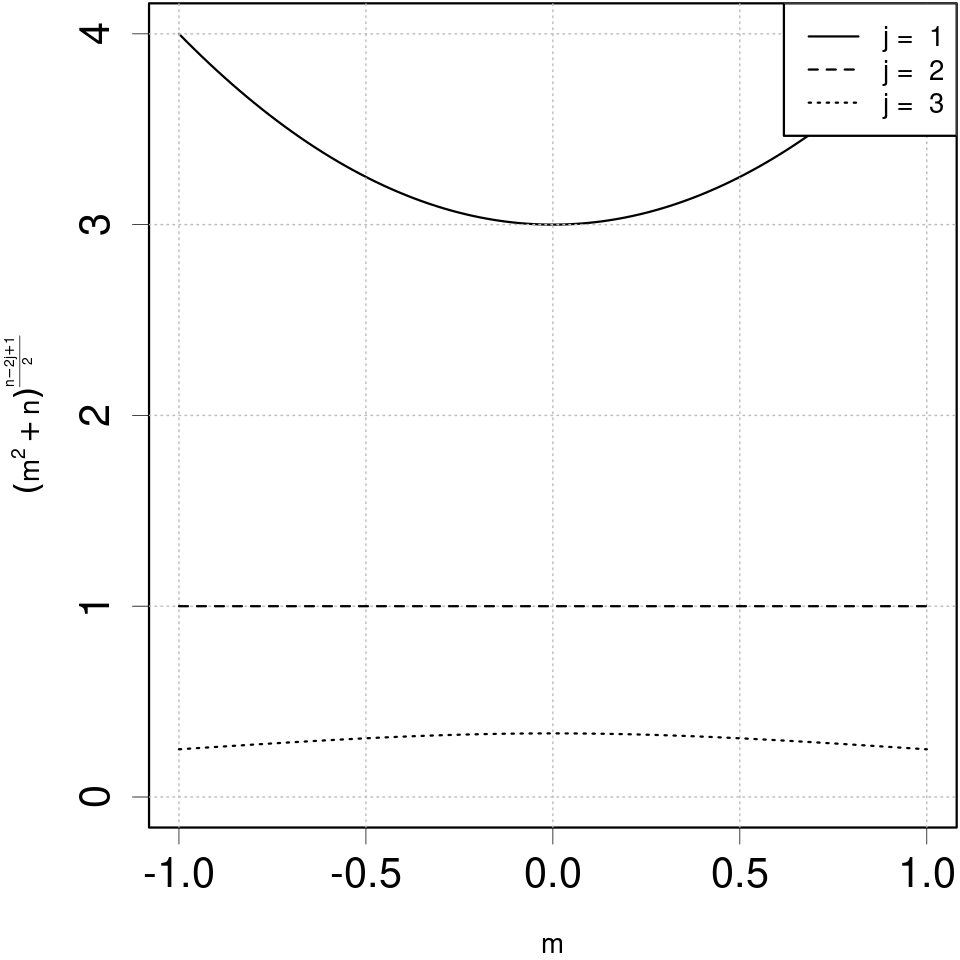}
\includegraphics[width=.495\textwidth]{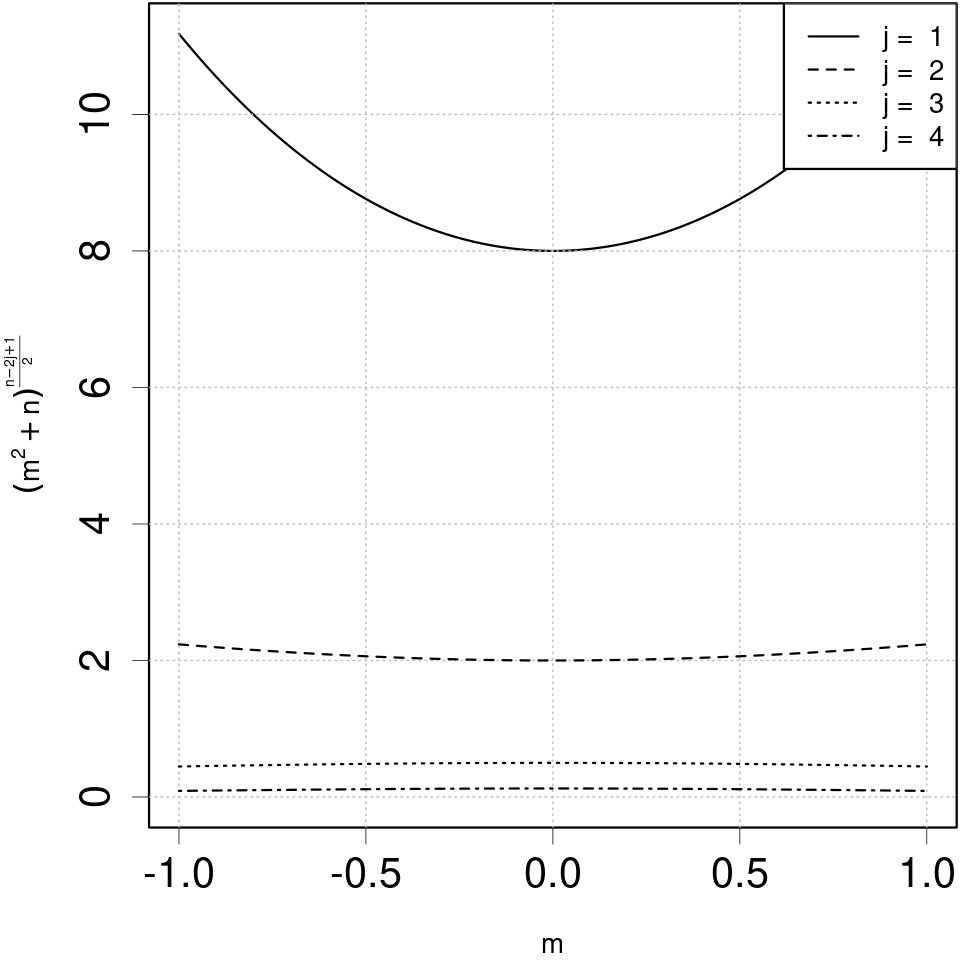}
\includegraphics[width=.495\textwidth]{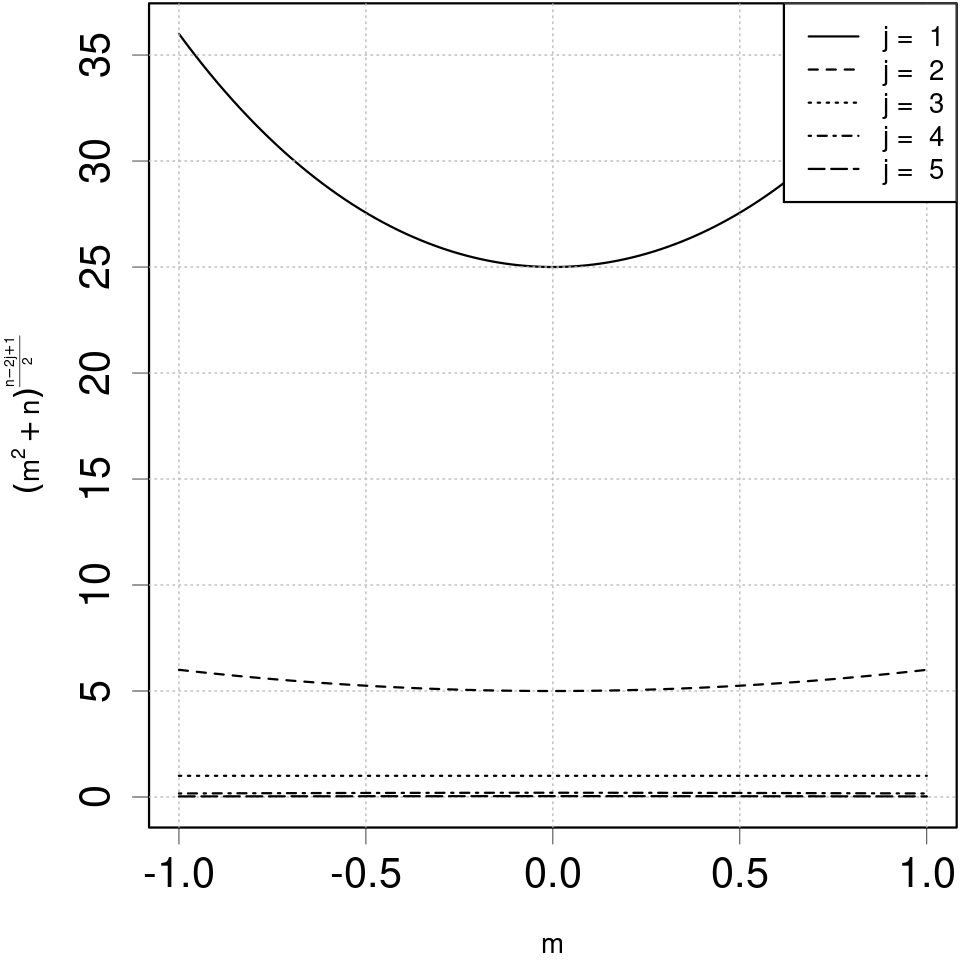}
\caption{\footnotesize From top to bottom and from left to right the function $(m^2+n)^{\frac{n-2j+1}{2}}$ for $n \in \{2,3,4,5\}$ degrees of freedom and $j\in\{1,\ldots,n\}$.}
\label{fig:constant}
\end{figure}


\section{On the $L_p$-quantiles for the Student $t$ distribution}\label{seq:lp-quant}
In this section we formally introduce the $L_p$-quantiles and present our second major contribution on identity relations between the $L_p$-quantiles of a Student $t$ distribution. 
 Specifically, consider the asymmetric power loss function $\ell_{p,\tau}: \mathbb{R}\to \mathbb{R}^+$, $\ell_{p,\tau}(x):=\tau(x_+)^p+(1-\tau)(x_-)^p$, where  $\tau\in(0,1)$. For $p\in\N$, the $L_p$-quantile of a random variable $Y$ at level $\tau\in(0,1)$, denoted by $\rho_{p,\tau}(Y)$, is defined as: 
\begin{equation}\label{eq:Lpquantilescore} 
\rho_{p,\tau}(Y):=\arg\min_{m\in\R}\E[\ell_{p,\tau}(Y-m)]=\arg\min_{m\in\R}\E[\tau((Y-m)_+)^p+(1-\tau)((Y-m)_-)^p],
\end{equation}
provided the expectation exists. For $p=1$, 
$\rho_{1,\tau}(Y)$ corresponds to the  quantile  at level $\tau$ of $Y$ and \eqref{eq:Lpquantilescore} has a unique solution only if the distribution of $Y$ is strictly increasing in a neighborhood of $\tau$. When $p=2$, $\rho_{2,\tau}(Y)$ corresponds to the $\tau$ level expectile introduced by~\cite{NeweyP87}. Expectiles have gained major attention in quantitative risk management as an important alternative to the well known VaR and ES risk measures. See for instance \cite{Emmer16} for a comparison of these three risk measures, \cite{Delbaen} and the reference therein for a characterization of expectiles as the unique example of a coherent elicitable risk measure and \cite{belliniDB17} for an empirical analysis of expectiles. Quantiles and expectiles with $\tau=\frac{1}{2}$ correspond respectively to the median and the mean of $Y$. 
 For $p\geq 2$, $L_p$-quantiles can be defined as the unique solution of the following first order condition: 
\begin{equation}\label{eq:Lpquantile}
 \tau \E\left[\left((Y-m)_+\right)^{p-1}\right]=(1-\tau)\E\left[\left((Y-m)_-\right)^{p-1}\right],\quad\tau\in(0,1).
\end{equation}
 {The main advantage of using~\eqref{eq:Lpquantile} is that it is well defined for random variables with $(p-1)$-th finite moments, while \eqref{eq:Lpquantilescore} requires also the $p$-th moment to be finite.}
 


In what follows, we show that for a random variable $Y$ having a Student $t$ distribution with $ n\in\mathbb{N}$ degrees of freedom the $L_{n-j+1}$-quantile and the $L_{j}$-quantile coincide for any level $\tau\in(0,1)$.


\begin{theorem}\label{mainth}
Let $Y$ be a Student $t$ random variable with $ n\in\mathbb{N} $ degrees of freedom. For $ j\in\{1,\ldots, n\}$ the following equation holds:
\begin{equation}\label{eq:equalitiLp}
\rho_{ n-j+1,\tau}(Y)=\rho_{j,\tau}(Y)\qquad\text{for any }\tau\in(0,1).
\end{equation}
\end{theorem}
\begin{proof} For $n=1$ there is nothing to prove. For $n\geq 2$ and any $\tau\in(0,1)$, the first order condition \eqref{eq:Lpquantile} says that  $\rho_{n -j+1,\tau}(Y)$ is the unique solution to
\begin{align}
 \tau \E\left[\left((Y-m)_+\right)^{n-j}\right]&=(1-\tau)\E\left[\left((Y-m)_-\right)^{n-j}\right]\notag\\
 &\Downarrow\notag\\
 \tau(m^2+n)^{\frac{n-2j+1}{2}}  \E\left[\left((Y-m)_+\right)^{j-1}\right]&=(1-\tau)(m^2+n)^{\frac{n-2j+1}{2}} \E\left[\left((Y-m)_-\right)^{j-1}\right]\notag\\
  &\Downarrow\notag\\
\tau \E\left[\left((Y-m)_+\right)^{j-1}\right]&=(1-\tau)\E\left[\left((Y-m)_-\right)^{j-1}\right],\label{focj}
\end{align}
where in the second step we used Theorem \ref{th:momentspp} and Corollary \ref{partial moments}. 
However, \eqref{focj} is satisfied if and only if $m=\rho_{j,\tau}(Y)$, from which we obtain the desired result {\mbox{$\rho_{n-j+1,\tau}(Y)=\rho_{j,\tau}(Y)$.}}
\end{proof}
From Theorem \ref{mainth} follows an important property of the $L_n$-quantiles of a Student $t$ distribution. Specifically, the $L_{n-j+1}$-quantile not only minimizes the asymmetric power loss function of order $n-j+1$ but it also minimizes the one of order $j$ for $ j\in\{1,\ldots, n\}$ and for any $\tau \in (0,1)$. Their symmetry characterization is reflected by the fact that, for a given number of degrees of freedom $n$, the $L_1$-quantile (quantile) and $L_n$-quantile coincide, the $L_2$-quantile (expectile) tallies with the $L_{n-1}$-quantile, and so forth for every pair of $L_{j}$- and $L_{n-j+1}$-quantiles. 

In the next corollary we extend this result to any affine transformation of the Student $t$ distribution.

\begin{corollary}\label{affine transformation}
Let $X$ be any affine transformation of a random variable $Y$ with Student $t$ distribution with $n\in\mathbb{N}$ degrees of freedom, that is $X \overset{d}{=} a + bY, a \in \R, b > 0$. For any $j\in\{1,\ldots, n\}$,
$$\rho_{ n -j+1,\tau}(X)=\rho_{j,\tau}(X)\qquad\text{for any }\tau\in(0,1).$$
\end{corollary}
\begin{proof}
The proof follows immediately by noting that $\rho_{ n -j+1,\tau}(a + bY)= a + b \rho_{n -j+1,\tau}(Y)$ for any $a \in \R,~ b > 0,~\tau\in(0,1)$ and $j\in\{1,\ldots, n\}$.
\end{proof}
Theorem \ref{mainth} and Corollary \ref{affine transformation} extend the work of \cite{Koenker92} who originally showed that the class of distributions for which the $L_2$-quantiles and $L_1$-quantiles coincide, corresponds to a rescaled Student $t$ distribution with 2 degrees of freedom. Moreover, in the special case $j=1$, the equality in Theorem \ref{mainth} reduces to the expression derived in \cite{BernardiBPStudent}, which shows that the $L_n$-quantile and the quantile coincide for any level $\tau\in(0,1)$. With our results we further generalize these works by providing a characterization of the symmetry of the $L_p$-quantiles for the Student $t$ distribution with $n$ degrees of freedom for any $p\in\{1,\ldots,n\}$.\\

 Figure \ref{fig:Lp} illustrates the behavior of the $L_j$-quantile for $j \in \{1,\dots,n\}$ of a Student $t$ distribution with $n=3$ (left) and $n=4$ (right) degrees of freedom with respect to $\tau \in (0,1)$. It is worth noticing that the $L_p$-quantiles are monotonically increasing, with a symmetry point for $\tau = \frac{1}{2}$ due to the symmetry of the Student $t$ centered around zero.
 Moreover, by looking at the plot on the left, 
 it is possible to see that $\rho_{1,\tau}$ (green) coincides with $\rho_{3,\tau}$ (red) for all $\tau\in(0, 1)$. Similarly, in the right-hand picture we have that $\rho_{1,\tau}$ (green) is equal to $\rho_{4,\tau}$ (pink) and $\rho_{2,\tau}$ (blue) coincides with $\rho_{3,\tau}$ (red) for all $\tau\in(0, 1)$.

\begin{figure}
\centering
\includegraphics[width=.495\textwidth]{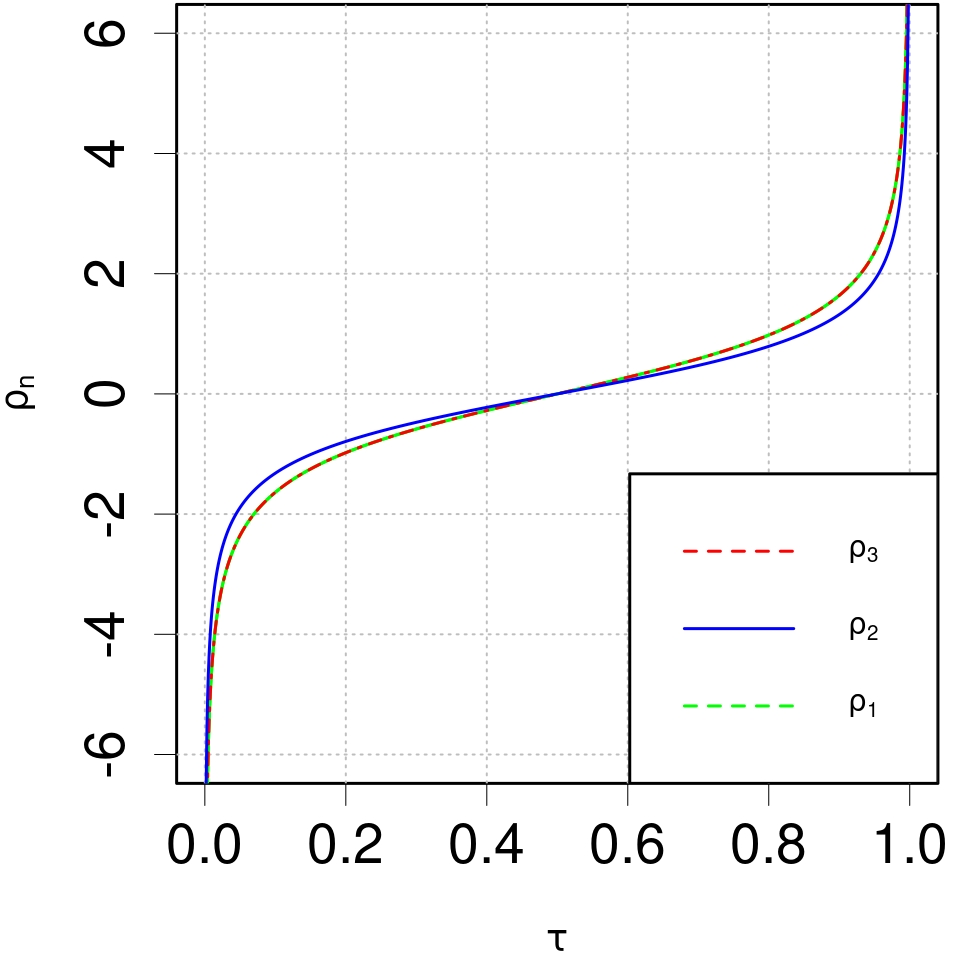}
\includegraphics[width=.495\textwidth]{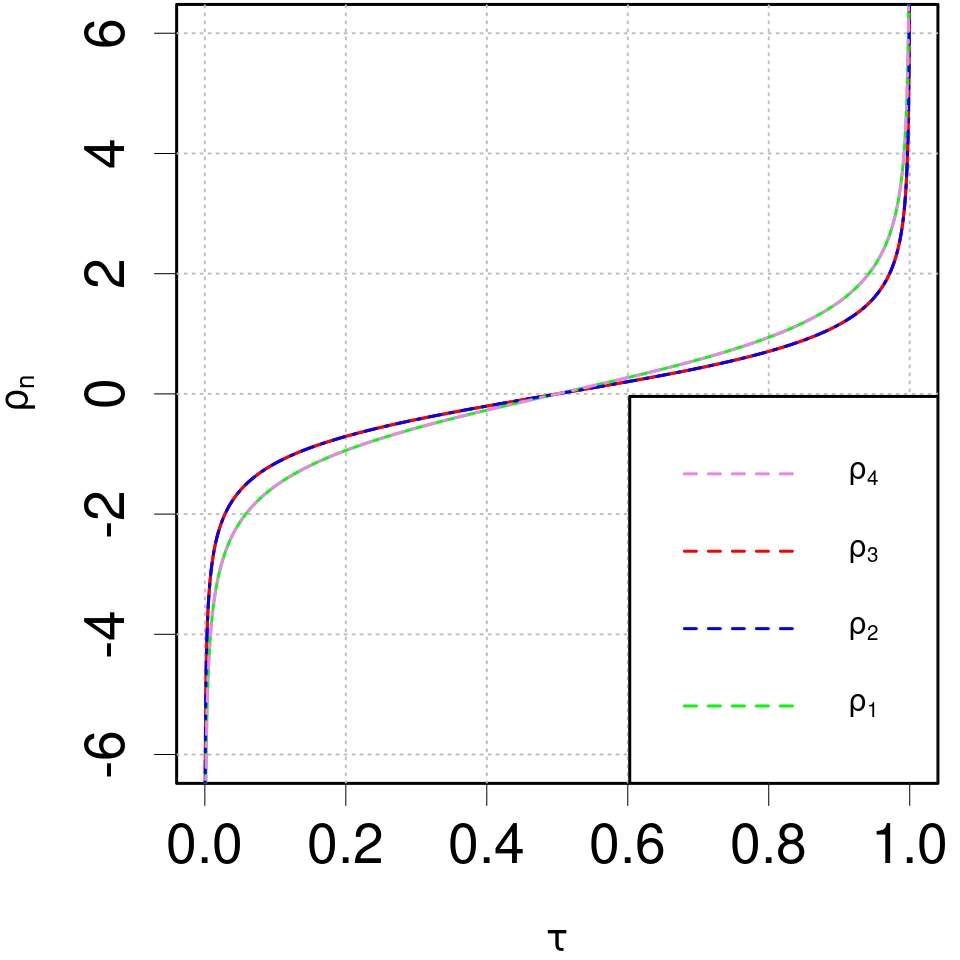}
\caption{\footnotesize $L_j$-quantiles for $j \in \{1,\dots,n\}$ of a Student $t$ distribution with $n=3$ (left) and $n=4$ (right) degrees of freedom with respect to $\tau \in (0,1)$. In both plots, the bi-colored dashed lines indicate that the $L_{n-j+1}$-quantile coincides with the $L_j$-quantile for $j \in \{1,\dots,n\}$ and for all $\tau\in(0, 1)$.}
\label{fig:Lp}
\end{figure}

\appendix
\section[Appendix]{Appendix}\label{Appendix}
\subsection{Proof Proposition \ref{Moments}}\label{Asec:Moments}
\textbf{Proposition \ref{Moments}} Let $Y$ be a random variable with Student $t$ distribution with $n\in\mathbb{N}$ degrees of freedom. For  $j\in\{0,\ldots,{n-1}\}$ and $m\in\R$:
\begin{equation}
\E[(Y-m)^j]=\left\{ 
\begin{tabular}{ll}
$-2mn^{\frac{j-1}{2}}\frac{\Gamma(\frac{j}{2}+1)\Gamma(\frac{n-j+1}{2})}{\sqrt{\pi}\Gamma(\frac{n}{2})}{_2}F_1\left(\frac{1-j}{2},\frac{n-j+1}{2},\frac{3}{2};-\frac{m^2}{n}\right)$& \qquad if $j~\textrm{is odd}$,\\
$n^{\frac{j}{2}}\frac{\Gamma(\frac{j+1}{2})\Gamma(\frac{n-j}{2})}{\sqrt{\pi}\Gamma(\frac{n}{2})}{_2}F_1\left(-\frac{j}{2},\frac{n-j}{2},\frac{1}{2};-\frac{m^2}{n}\right)$& \qquad if $j~\textrm{is even},$\\
\end{tabular}
\right. 
\end{equation}
where $_{2}{F}_1 (a, b, c; z)$ is the hypergeometric function defined in \eqref{eq:2f1}.
\begin{proof}
By using the binomial expansion and the linearity of the expectation we can decompose the term $\mathbb{E}{[(Y-m)^{j}]}$, as:
$$
\mathbb{E}{[(Y-m)^{j}]}=\sum_{k=0}^{j}\binom{j}{k} \E[Y^{j-k}](-m)^{k}.
$$ 
{Note that for $j\in\{0,2,\ldots,n-1\}$ and $k\in\{0,\ldots,j\}$ all the moments $ \E[Y^{j-k}]$ are finite.} 
We first assume that $j$ is even. The term $ \E[Y^{j-k}]$ is not zero only for $k$ even, therefore  the $k$ index is changed  $k\mapsto 2k$ to obtain: 
$$
\mathbb{E}{[(Y-m)^{j}]}=\sum_{k=0}^{\frac{j}{2}}\binom{j}{2k} \E[Y^{j-2k}](-m)^{2k}.
$$
One can now substitute the values of $\E[Y^{j-2k}]$ from \eqref{eq:rawmoments} and express the binomial coefficient in terms of the Gamma function to get:
\begin{align}\mathbb{E}{[(Y-m)^{j}]}&=\sum_{k=0}^{\frac{j}{2}}\frac{\Gamma(j+1)}{\Gamma(2k+1)\Gamma(j-2k+1)} \frac{\Gamma(\frac{j-2k+1}{2})\Gamma(\frac{n-j+2k}{2})}{\sqrt{\pi}\Gamma(\frac{n}{2})}n^{\frac{j}{2}}\left(\frac{m^2}{n}\right)^{k}\notag\\
&=\frac{n^{\frac{j}{2}}\Gamma(\frac{n-j}{2})}{\sqrt{\pi}\Gamma(\frac{n}{2})}\sum_{k=0}^{\frac{j}{2}}\frac{\Gamma(j+1)}{\Gamma(2k+1)\Gamma(j-2k+1)} \frac{\Gamma(\frac{j-2k+1}{2})\Gamma(\frac{n-j}{2}+k)}{\Gamma(\frac{n-j}{2})}\left(\frac{m^2}{n}\right)^{k},\label{eq:jeven}
\end{align}
where in the last equation we multiplied and divided by $\Gamma(\frac{n-j}{2})$. One can now apply the Legendre duplication formula for the Gamma function, $\Gamma(z)\Gamma(z+\frac{1}{2})=2^{1-2z}\sqrt{\pi}\Gamma(2z),~z\in\mathbb{R}^+$, to the following terms:
\begin{align*}
&\frac{\Gamma(\frac{j-2k+1}{2})}{\Gamma(j-2k+1)}=\frac{2^{-j+2k}\sqrt{\pi}}{\Gamma(\frac{j-2k}{2}+1)};\\
&\Gamma(j+1)=\frac{\Gamma(\frac{j+1}{2})\Gamma(\frac{j}{2}+1)}{2^{-j}\sqrt{\pi}};\\
&\Gamma(2k+1)=\frac{{\Gamma(k+\frac{1}{2})\Gamma(k+1)}}{2^{-2k}\sqrt{\pi}}.
\end{align*}
By substituting the above formulas in \eqref{eq:jeven}, recalling the definition of the Pochhammer symbol $(x)_k=\frac{\Gamma(x+k)}{\Gamma(x)}$ and that $\sqrt{\pi}=\Gamma(\frac{1}{2})$, one gets
\begin{align*}
\E[(Y-m)^{j}]&=\frac{n^{\frac{j}{2}}\Gamma(\frac{n-j}{2})}{\sqrt{\pi}\Gamma(\frac{n}{2})}\sum_{k=0}^{\frac{j}{2}}\frac{\Gamma(\frac{j+1}{2})\Gamma(\frac{j}{2}+1)}{\frac{{\Gamma(k+\frac{1}{2})\Gamma(k+1)}}{\Gamma(\frac{1}{2})}\Gamma(\frac{j}{2}-k+1)}\frac{\Gamma(\frac{n-j}{2}+k)}{\Gamma(\frac{n-j}{2})}(-1)^k\left(-\frac{m^2}{n}\right)^{k}\\
&=\frac{n^{\frac{j}{2}}\Gamma(\frac{j+1}{2})\Gamma(\frac{n-j}{2})}{\sqrt{\pi}\Gamma(\frac{n}{2})}\sum_{k=0}^{\frac{j}{2}}(-1)^k\binom{\frac{j}{2}}{k}\frac{\left(\frac{n-j}{2}\right)_k}{\left(\frac{1}{2}\right)_k}\left(-\frac{m^2}{n}\right)^{k}\\
&= n^{\frac{j}{2}} \frac{\Gamma(\frac{n-j}{2})}{\sqrt{\pi}} \frac{\Gamma(\frac{j+1}{2})}{\Gamma(\frac{n}{2})} {_{2}{F}_1} \left(- \frac{j}{2}, \frac{n-j}{2}, \frac{1}{2}; - \frac{m^2}{n}\right).
\end{align*}
{Since the first term of the hypergeometric function is a nonpositive integer, the convergence of the function is guaranteed for any value of $m\in\R$ and $n\in\mathbb{N}$.}
\\
Consider now the case $j$ odd.
The term $ \E[Y ^{j-k}]$ is not zero only for $k$ odd, therefore the $k$ index is changed  $k\mapsto 2k+1$ to obtain: 
$$
\mathbb{E}{[(Y-m)^{j}]}=\sum_{k=0}^{\frac{j-1}{2}}\binom{j}{2k+1} \E[Y^{j-2k-1}](-m)^{2k+1}.
$$
One can now substitute the values of $\E[Y^{j-2k-1}]$ from \eqref{eq:rawmoments} and express the binomial coefficient in terms of the Gamma function to get:
\begin{align}\mathbb{E}{[(Y-m)^{j}]}&=-m\sum_{k=0}^{\frac{j-1}{2}}\frac{\Gamma(j+1)}{\Gamma(2k+2)\Gamma(j-2k)} \frac{\Gamma(\frac{j-2k}{2})\Gamma(\frac{n-j+2k+1}{2})}{\sqrt{\pi}\Gamma(\frac{n}{2})}n^{\frac{j-1}{2}}\left(\frac{m^2}{n}\right)^{k}\notag\\
&=-m\frac{n^{\frac{j-1}{2}}\Gamma(\frac{n-j+1}{2})}{\sqrt{\pi}\Gamma(\frac{n}{2})}\sum_{k=0}^{\frac{j-1}{2}}\frac{\Gamma(j+1)}{\Gamma(2k+2)\Gamma(j-2k)} \frac{\Gamma(\frac{j-2k}{2})\Gamma(\frac{n-j+2k+1}{2})}{\Gamma(\frac{n-j+1}{2})}\left(\frac{m^2}{n}\right)^{k},\label{eq:jodd}
\end{align}
where in the last equation we multiplied and divided by $\Gamma(\frac{n-j+1}{2})$. Again, we apply the Legendre duplication formula for the Gamma function 
to the following terms:
\begin{align*}
&\frac{\Gamma(\frac{j-2k}{2})}{\Gamma(j-2k)}=\frac{2^{1-j+2k}\sqrt{\pi}}{\Gamma(\frac{j-2k+1}{2})};\\
&\Gamma(j+1)=\frac{\Gamma(\frac{j+1}{2})\Gamma(\frac{j}{2}+1)}{2^{-j}\sqrt{\pi}};\\
&\Gamma(2k+2)=\frac{{\Gamma(k+1)\Gamma(k+\frac{3}{2})}}{2^{-2k-1}\sqrt{\pi}}.
\end{align*}
By substituting the above formulas in \eqref{eq:jodd}, one gets
\begin{align*}
\E[(Y-m)^{j}]&=-m\frac{n^{\frac{j-1}{2}}\Gamma(\frac{n-j+1}{2})}{\sqrt{\pi}\Gamma(\frac{n}{2})}\sum_{k=0}^{\frac{j-1}{2}}\frac{{\Gamma(\frac{j-1}{2}+1)\Gamma(\frac{j}{2}+1)}}{{\Gamma(k+1)\Gamma(k+\frac{3}{2})}}\frac{\sqrt{\pi}}{\Gamma(\frac{j-2k+1}{2})} \frac{\Gamma(\frac{n-j+2k+1}{2})}{\Gamma(\frac{n-j+1}{2})}\left(\frac{m^2}{n}\right)^{k}\\
=&-2m\frac{n^{\frac{j-1}{2}}\Gamma(\frac{j}{2}+1)\Gamma(\frac{n-j+1}{2})}{\sqrt{\pi}\Gamma(\frac{n}{2})}\sum_{k=0}^{\frac{j-1}{2}}(-1)^k\binom{\frac{j-1}{2}}{k}\frac{\sqrt{\pi}}{2{\Gamma(k+\frac{3}{2})}}\left(\frac{n-j+1}{2}\right)_k\left(-\frac{m^2}{n}\right)^{k},
&\end{align*}
where in the last step we multiplied and divided by $2$. Recalling that $\Gamma(\frac{3}{2})=\frac{\sqrt{\pi}}{2}$, we find: 
$$
\E[(Y-m)^{j}]=-2m\frac{n^{\frac{j-1}{2}}\Gamma(\frac{j}{2}+1)\Gamma(\frac{n-j+1}{2})}{\sqrt{\pi}\Gamma(\frac{n}{2})}{_2}F_1\left(-\frac{j-1}{2},\frac{n-j+1}{2},\frac{3}{2};-\frac{m^2}{n}\right)
$$
as desired.
{Since the first term of the hypergeometric function is a nonpositive integer, the convergence of the function is guaranteed for any value of $m\in\R$ and $n\in\mathbb{N}$}.
\end{proof}
\subsection{Proof of Lemma \ref{lemma}}\label{Asec:lemma}
\textbf{Lemma \ref{lemma}} Let $Y$ be a random variable with Student $t$ distribution with $n\in\mathbb{N}$ degrees of freedom and let $j\in\{1,\ldots,{n}\}$:
\begin{enumerate}\item[(1)] For $m\geq 0$ the following equation holds:
\begin{equation}
\E[((Y-m)_+)^{n-j}]=(m^2+n)^{\frac{n-2j+1}{2}}\E[((Y-m)_+)^{j-1}].
\end{equation}
\item[(2)] For $m<0$
\begin{equation}
\E[((Y-m)_-)^{n-j}]=(m^2+n)^{\frac{n-2j+1}{2}}\E[((Y-m)_-)^{j-1}].
\end{equation}
\end{enumerate}
\begin{proof} We first note that from \eqref{eq:uppermom} and \eqref{eq:lowermom}, the upper and lower partial moments of order $j\in\{0,\ldots,n-1\}$ are
\begin{align*}
\E[((Y-m)_+)^{j}]&=\int_m^{+\infty} (y-m)^{j}\textrm{d}F_Y(y)\quad \textrm{and } \\
\E[((Y-m)_-)^{j}]&=(-1)^j\int_{-\infty}^{m} (y-m)^{j}\textrm{d}F_Y(y).\\
\end{align*}

\begin{enumerate}\item[(1)] For $m\geq 0$, formula 3.254(2) in  \cite{zwillinger2015table} implies:
\begin{equation}\label{eq:zwillinger}
\int_m^{+\infty} (y-m)^{j}\textrm{d}F_Y(y) = C_nm^{j-n} B(n-j,j+1) {_{2}{F}_1} \left(\frac{n-j}{2},\frac{n-j+1}{2}, \frac{n+2}{2}; -\frac{n}{m^2}\right),
\end{equation}
where  $C_n=n^{\frac{n}{2}}\frac{\Gamma\left(\frac{n+1}{2}\right)}{\sqrt{\pi}\Gamma\left(\frac{n}{2}\right)}$. It follows that {for $j=1,\dots,n$}
\begin{equation*}
\int_m^{+\infty} (y-m)^{j-1}\textrm{d}F_Y(y) = C_nm^{j-n-1} B(n-j+1,j) {_{2}{F}_1} \left(\frac{n-j+1}{2}, \frac{n-j+2}{2}, \frac{n+2}{2}; -\frac{n}{m^2}\right)
\end{equation*}
and
\begin{equation}\label{eq:int}
\int_m^{+\infty} (y-m)^{n-j}\textrm{d}F_Y(y) =C_n m^{-j} B(n-j+1,j) {_{2}{F}_1} \left(\frac{j}{2}, \frac{j+1}{2}, \frac{n+2}{2}; -\frac{n}{m^2}\right).
\end{equation}
We then apply Euler's transformation to the hypergeometric function in \eqref{eq:int} to get
{\small
\begin{align*}
\int_m^{+\infty} (y-m)^{n-j}\textrm{d}F_Y(y) &= C_nm^{-j} B(n-j+1,j) m^{-n+2j-1}(m^2+n)^{\frac{n-2j+1}{2}}{_{2}{F}_1} \left(\frac{n-j+2}{2}, \frac{n-j+1}{2}, \frac{n+2}{2}; -\frac{n}{m^2}\right)\\
&=(m^2+n)^{\frac{n-2j+1}{2}}\int_m^{+\infty} (y-m)^{j-1}\textrm{d}F_Y(y),
\end{align*}
}
where in the last step  we used the symmetry of the hypergeometric function.
\item[(2)]  Consider the case $m<0$. For $j\in\{0,\ldots,{n-1}\}$ we apply a change of variable: $x=-y$ and use again formula 3.254(2) in  \cite{zwillinger2015table}, to get
\begin{align*}
\int_{-\infty}^{m}(y-m)^{j}\textrm{d}F_Y(y)&=(-1)^{j}\int_{-m}^{+\infty}(x-(-m))^{j}\textrm{d}F_Y(x)\\
&=(-1)^{j}(-1)^{j-n}C_nm^{j-n} B(n-j,j+1) {_{2}{F}_1} \left(\frac{n-j}{2},\frac{n-j+1}{2}, \frac{n+2}{2}; -\frac{n}{m^2}\right)\\
&=(-1)^{-n}C_nm^{j-n} B(n-j,j+1) {_{2}{F}_1} \left(\frac{n-j}{2},\frac{n-j+1}{2}, \frac{n+2}{2}; -\frac{n}{m^2}\right).\\
\end{align*}
It follows that {for $j=1,\dots,n$}
\begin{equation*}
\int_{-\infty}^m (y-m)^{j-1}\textrm{d}F_Y(y) = (-1)^{-n}C_nm^{j-n-1} B(n-j+1,j) {_{2}{F}_1} \left(\frac{n-j+1}{2}, \frac{n-j+2}{2}, \frac{n+2}{2}; -\frac{n}{m^2}\right)
\end{equation*}
and
\begin{align*}
&\int_{-\infty}^{m} (y-m)^{n-j}\textrm{d}F_Y(y) =\\
&= (-1)^{-n}C_nm^{-j}B(n-j+1,j) {_{2}{F}_1} \left(\frac{j}{2}, \frac{j+1}{2}, \frac{n+2}{2}; -\frac{n}{m^2}\right)\\
&=(-1)^{-n+2j-1} (m^2+n)^{\frac{n-2j+1}{2}}(-1)^{-n}C_nm^{-n+j-1}B(n-j+1,j) {_{2}{F}_1} \left(\frac{n-j+2}{2}, \frac{n-j+1}{2}, \frac{n+2}{2}; -\frac{n}{m^2}\right)\\
&=(-1)^{n+1}(m^2+n)^{\frac{n-2j+1}{2}}\int_{-\infty}^m (y-m)^{j-1}\textrm{d}F_Y(y),
\end{align*}
where in the third equality we applied again the Euler's transformation and the symmetry of the hypergeometric function.
\end{enumerate}
\end{proof}

\subsection{Proof Theorem \ref{th:moments}}\label{Asec:thmoments}
\textbf{Theorem \ref{th:moments}}
%
 Let $Y$ be a random variable with Student $t$ distribution with $n\in\mathbb{N}$ degrees of freedom. For $j\in\{1,\ldots,n\}$ and $m\in\R$, the following equation for the central moments about the point $m$ holds:
\begin{enumerate}\item For $n$ odd:
\begin{equation}
\mathbb{E}[(Y-m)^{n-j}]=(m^2+n)^{\frac{n-2j+1}{2}}\mathbb{E}[(Y-m)^{j-1}].
\end{equation}
\item For $n$ even:
\begin{align}
\mathbb{E}[(Y-m)^{n-j}]&={(m^2+n)^{\frac{n-2j+1}{2}}\left[-\mathbb{E}[(Y-m)^{j-1}]+2\mathbb{E}[((Y-m)_+)^{j-1}] \right]}.
\end{align}
\end{enumerate}

\begin{proof} 
\begin{enumerate}
\item We know that $n$ is odd and first consider the case $j$ odd, from which it follows that $n-j$ and $j-1$ are even. From the expression \eqref{eq:cenmom} for an even power, we get:
$$
\mathbb{E}{[(Y-m)^{j-1}]}=n^{\frac{j-1}{2}} \frac{\Gamma(\frac{n-j+1}{2})\Gamma(\frac{j}{2})}{\sqrt{\pi}\Gamma(\frac{n}{2})} {_{2}{F}_1} \left(- \frac{j-1}{2}, \frac{n-j+1}{2}, \frac{1}{2}; - \frac{m^2}{n}\right)$$ 
and 
$$
\mathbb{E}{[(Y-m)^{n-j}]}=n^{\frac{n-j}{2}} \frac{\Gamma(\frac{n-j+1}{2})\Gamma(\frac{j}{2})}{\sqrt{\pi}\Gamma(\frac{n}{2})} {_{2}{F}_1} \left(- \frac{n-j}{2}, \frac{j}{2}, \frac{1}{2}; - \frac{m^2}{n}\right).
$$
Using the symmetry of the hypergeometric function and applying the Euler's transformation to ${_{2}{F}_1} \left(\frac{j}{2}, - \frac{n-j}{2}, \frac{1}{2}; - \frac{m^2}{n}\right)$ concludes the proof.

For $j$ even, $n-j$ and $j-1$ are odd. Using again the expression \eqref{eq:cenmom} for an odd power, one gets
$$
\mathbb{E}{[(Y-m)^{j-1}]}=\frac{-2m n^{\frac{j-2}{2}}\Gamma(\frac{j+1}{2})\Gamma(\frac{n-j+2}{2})}{\sqrt{\pi}\Gamma(\frac{n}{2})}{_{2}{F}_1} \left(- \frac{j-2}{2}, \frac{n-j+2}{2}, \frac{3}{2}; - \frac{m^2}{n}\right)
$$
and 
$$
\mathbb{E}{[(Y-m)^{n-j}]}=\frac{-2m n^{\frac{n-j-1}{2}}\Gamma(\frac{j+1}{2})\Gamma(\frac{n-j+2}{2})}{\sqrt{\pi}\Gamma(\frac{n}{2})} {_{2}{F}_1}(-\frac{n-j-1}{2},\frac{j+1}{2},\frac{3}{2};-\frac{m^2}{n}).
$$
Using again the symmetry of the hypergeometric function and applying  the Euler's transformation to ${_{2}{F}_1}(\frac{j+1}{2},-\frac{n-j-1}{2},\frac{3}{2};-\frac{m^2}{n})$ concludes the proof.
\item For $n$ even, we first consider $j$ even. Since $n-j$ is even and $j-1$ is odd, using \eqref{eq:cenmom} we find
\begin{equation}\label{eq:2}
\mathbb{E}{[(Y-m)^{n-j}]}=\frac{ n^{\frac{n-j}{2}}\Gamma(\frac{n-j+1}{2})\Gamma(\frac{j}{2})}{\sqrt{\pi}\Gamma(\frac{n}{2})} {_{2}{F}_1}\left(-\frac{n-j}{2},\frac{j}{2},\frac{1}{2};-\frac{m^2}{n}\right)
\end{equation}
and 
\begin{align}
&(m^2+n)^{\frac{n-2j+1}{2}}\mathbb{E}[(Y-m)^{j-1}]=\notag\\&\qquad=(m^2+n)^{\frac{n-2j+1}{2}}(-2mn^{\frac{j-2}{2}})\frac{\Gamma\left(\frac{j+1}{2}\right)\Gamma\left(\frac{n-j+2}{2}\right)}{\sqrt{\pi}\Gamma\left(\frac{n}{2}\right)} {_{2}{F}_1}\left(-\frac{j-2}{2},\frac{n-j+2}{2},\frac{3}{2};-\frac{m^2}{n}\right)\notag\\
&\qquad={-2mn^{\frac{n-j-1}{2}}}\frac{\Gamma\left(\frac{j+1}{2}\right)\Gamma\left(\frac{n-j+2}{2}\right)}{\sqrt{\pi}\Gamma\left(\frac{n}{2}\right)} {_{2}{F}_1}\left(\frac{j+1}{2},-\frac{n-j-1}{2},\frac{3}{2};-\frac{m^2}{n}\right),\label{eq:3}
\end{align}
where in the last equality we used the Euler's transformation for the hypergeometric function. {Furthermore, for $m\geq0$ Lemma \ref{lemma}  and  formula 3.254(2) in  \cite{zwillinger2015table} provide
\begin{align}
2(m^2+n)^{\frac{n-2j+1}{2}}\E[((Y-m)_+)^{j-1}]&=2\E[((Y-m)_+)^{n-j}]\notag\\
&=2C_n m^{-j} B(n-j+1,j) {_{2}{F}_1} \left(\frac{j}{2}, \frac{j+1}{2}, \frac{n+2}{2}; -\frac{n}{m^2}\right)\label{eq:1}.
\end{align}
Our goal is then to show that
\begin{equation}\label{eq:toprove}
2\E[((Y-m)_+)^{n-j}]=\mathbb{E}{[(Y-m)^{n-j}]}+(m^2+n)^{\frac{n-2j+1}{2}}\mathbb{E}{[(Y-m)^{j-1}]}
\end{equation}
using the property in \eqref{property} of the hypergeometric functions.
From  \eqref{eq:2}, \eqref{eq:3} and  \eqref{eq:1}, \eqref{eq:toprove} becomes
\begin{align*}
&{_{2}{F}_1}\left(\frac{j}{2},\frac{j+1}{2},\frac{n+2}{2};-\frac{n}{m^2}\right)=\\
&\frac{m^{j}}{2C_nB(n-j+1,j)}\left[\frac{ n^{\frac{n-j}{2}}\Gamma(\frac{n-j+1}{2})\Gamma(\frac{j}{2})}{\sqrt{\pi}\Gamma(\frac{n}{2})} {_{2}{F}_1}\left(\frac{j}{2},-\frac{n-j}{2},\frac{1}{2};-\frac{m^2}{n}\right)\right]+\\
&\frac{m^{j}}{2C_nB(n-j+1,j)}\left[{-2mn^{\frac{n-j-1}{2}}}\frac{\Gamma\left(\frac{j+1}{2}\right)\Gamma\left(\frac{n-j+2}{2}\right)}{\sqrt{\pi}\Gamma\left(\frac{n}{2}\right)} {_{2}{F}_1}\left(\frac{j+1}{2},-\frac{n-j-1}{2},\frac{3}{2};-\frac{m^2}{n}\right)
\right].\notag\end{align*}
We first note that for  $a=\frac{j}{2},~b=\frac{j+1}{2},~c=\frac{n+2}{2}$ and $z=-\frac{n}{m^2}$ the hypergeometric functions above have the correct terms as the ones required by the property in \eqref{property}. Therefore, we are left to show that the multiplicative coefficients {associated to the two hypergeometric functions on the right-hand side} correspond. For the first multiplicative term, the desired result follows from  
\begin{align*}
\frac{m^{j}}{2C_nB(n-j+1,j)}\frac{ n^{\frac{n-j}{2}}\Gamma(\frac{n-j+1}{2})\Gamma(\frac{j}{2})}{\sqrt{\pi}\Gamma(\frac{n}{2})}&=
\frac{m^{j}\sqrt{\pi}\Gamma\left(\frac{n}{2}\right)\Gamma\left({n+1}\right)}{2\Gamma\left(\frac{n+1}{2}\right)n^{\frac{n}{2}}\Gamma\left(n-j+1\right)\Gamma\left(j\right)}\frac{ n^{\frac{n-j}{2}}\Gamma\left(\frac{n-j+1}{2}\right)\Gamma\left(\frac{j}{2}\right)}{\sqrt{\pi}\Gamma\left(\frac{n}{2}\right)}\\
&=\frac{\left(\frac{n}{m^2}\right)^{-\frac{j}{2}}\Gamma\left({n+1}\right)\Gamma\left(\frac{n-j+1}{2}\right)\Gamma\left(\frac{j}{2}\right)}{2\Gamma\left(\frac{n+1}{2}\right)\Gamma\left(n-j+1\right)\Gamma\left(j\right)}\\
&={\frac{\Gamma(\frac{n+2}{2}) \Gamma(\frac{1}{2})}{\Gamma(\frac{j+1}{2}) \Gamma(\frac{n-j+2}{2})} \left(\frac{n}{m^2}\right)^{-\frac{j}{2}}},
\end{align*}
where in the last step we used the Legendre duplication formula for the Gamma function. {Moving on to the second multiplicative coefficient we have:
\begin{align*}
\frac{m^{j} (-2mn^{\frac{n-j-1}{2}})}{2C_nB(n-j+1,j)}\frac{\Gamma\left(\frac{j+1}{2}\right)\Gamma\left(\frac{n-j+2}{2}\right)}{\sqrt{\pi}\Gamma\left(\frac{n}{2}\right)}&=-\frac{m^{j+1} n^{\frac{n-j-1}{2}} \sqrt{\pi}\Gamma\left(\frac{n}{2}\right)\Gamma\left({n+1}\right)}{\Gamma\left(\frac{n+1}{2}\right)n^{\frac{n}{2}}\Gamma\left({n-j+1}\right)\Gamma\left({j}\right)}\frac{\Gamma\left(\frac{j+1}{2}\right)\Gamma\left(\frac{n-j+2}{2}\right)}{\sqrt{\pi}\Gamma\left(\frac{n}{2}\right)}\\
&=-\frac{\left(\frac{n}{m^2}\right)^{-\frac{j+1}{2}}\Gamma\left({n+1}\right)\Gamma\left(\frac{j+1}{2}\right)\Gamma\left(\frac{n-j+2}{2}\right)}{\Gamma\left(\frac{n+1}{2}\right)\Gamma\left(n-j+1\right)\Gamma\left(j\right)}\\
&=\frac{\Gamma(\frac{n+2}{2}) \Gamma(-\frac{1}{2})}{\Gamma(\frac{j}{2}) \Gamma(\frac{n-j+1}{2})} \left(\frac{n}{m^2}\right)^{-\frac{j+1}{2}},
\end{align*}
where in the last equality we used that $-2\sqrt{\pi} = \Gamma\left(-\frac{1}{2}\right)$.
}}

{When} $m<0$, we first note that,
\begin{align*}\mathbb{E}[(Y-m)^{n-j}]&={(m^2+n)^{\frac{n-2j+1}{2}}\left[-\mathbb{E}[(Y-m)^{j-1}]+2\mathbb{E}[((Y-m)_+)^{j-1}] \right]}\\
&={(m^2+n)^{\frac{n-2j+1}{2}}\left[\mathbb{E}[(Y-m)^{j-1}]+2(-1)^{j-1}\mathbb{E}[((Y-m)_-)^{j-1}] \right]}.\\
\end{align*}
Lemma \ref{lemma} then provides {
\begin{align}
2(m^2+n)^{\frac{n-2j+1}{2}}\E[((Y-m)_-)^{j-1}]&=2\E[((Y-m)_-)^{n-j}]\notag\\
&=2C_nm^{-j}B(n-j+1,j){_{2}{F}_1}\left(\frac{j}{2},\frac{j+1}{2},\frac{n+2}{2};-\frac{n}{m^2}\right)\label{eq:1mneg}.
\end{align}}
Since $j$ is even,  our goal is, as before,  to show that
\begin{equation}\label{eq:toprovemneg}
2 \E[((Y-m)_-)^{j-1}]=\mathbb{E}{[(Y-m)^{n-j}]}-(m^2+n)^{\frac{n-2j+1}{2}}\mathbb{E}{[(Y-m)^{j-1}]}
\end{equation}
using {the property in \eqref{property} of the hypergeometric functions. From  \eqref{eq:2}, \eqref{eq:3} and  \eqref{eq:1mneg}, \eqref{eq:toprovemneg} becomes
\begin{align*}
&{_{2}{F}_1}\left(\frac{j}{2},\frac{j+1}{2},\frac{n+2}{2};-\frac{n}{m^2}\right)=\\
&\frac{m^{j}}{2C_nB(n-j+1,j)}\left[\frac{ n^{\frac{n-j}{2}}\Gamma(\frac{n-j+1}{2})\Gamma(\frac{j}{2})}{\sqrt{\pi}\Gamma(\frac{n}{2})} {_{2}{F}_1}\left(\frac{j}{2},-\frac{n-j}{2},\frac{1}{2};-\frac{m^2}{n}\right)\right]-\\
&\frac{m^{j}}{2C_nB(n-j+1,j)}\left[-2mn^{\frac{n-j-1}{2}}\frac{\Gamma\left(\frac{j+1}{2}\right)\Gamma\left(\frac{n-j+2}{2}\right)}{\sqrt{\pi}\Gamma\left(\frac{n}{2}\right)} {_{2}{F}_1}\left(\frac{j+1}{2},-\frac{n-j-1}{2},\frac{3}{2};-\frac{m^2}{n}\right)
\right].\notag\end{align*}
The proof follows analogously to the case $m\geq 0$.\\

For the case $j$ odd, $n-j$ is odd and $j-1$ is even. The proof is obtained as before exploiting Lemma \ref{lemma} and the hypergeometric function property in \eqref{property} applied to the cases $m\geq0$ and $m<0$.} 
\end{enumerate}

\end{proof}

\bibliographystyle{chicago}
 \bibliography{bibLpquantile_short}

\end{document}